\documentclass[11pt, a4paper]{article}
\usepackage[utf8]{inputenc}

\usepackage{amsmath,amssymb,amsthm}
\usepackage{mathtools}
\usepackage{thmtools}
\usepackage{thm-restate}

\usepackage{fullpage}
\usepackage{xcolor}
\usepackage[nocompress, noadjust]{cite}
\usepackage{hyperref, cleveref}
\hypersetup{hidelinks}
\usepackage[skip=7pt plus2pt]{parskip}

\providecommand{\keywords}[1]{\emph{Keywords: } #1.}

\usepackage{mleftright}
\mleftright

\newtheorem{theorem}{Theorem}[section]
\newtheorem{proposition}[theorem]{Proposition}
\newtheorem{lemma}[theorem]{Lemma}
\newtheorem{claim}{Claim}
\newtheorem{question}[theorem]{Question}
\newtheorem{corollary}[theorem]{Corollary}
\newtheorem{observation}[theorem]{Observation}
\newtheorem{conjecture}[theorem]{Conjecture}

\theoremstyle{remark}
\newtheorem{case}{Case}
\makeatletter\@addtoreset{case}{theorem}\@addtoreset{case}{lemma}\@addtoreset{case}{corollary}\makeatother

\title{A note on inverting the dijoin of oriented graphs}
\author{Natalie Behague\thanks{Mathematics Institute, University of Warwick, Coventry, UK. Email: \href{mailto:natalie.behague@warwick.ac.uk}{natalie.behague@warwick.ac.uk}. Supported by a PIMS postdoctoral fellowship while at the University of Victoria.}   \and Tom Johnston\thanks{Corresponding author. School of Mathematics, University of Bristol, Bristol, BS8\thinspace1UG, UK and Heilbronn Institute for Mathematical Research, Bristol, BS8\thinspace1UG, UK. Email: \href{mailto:tom.johnston@bristol.ac.uk}{tom.johnston@bristol.ac.uk}.}\and Natasha Morrison\footnotemark[3]\thanks{Department of Mathematics and Statistics, University of Victoria, Canada.} \thanks{Research supported by NSERC Discovery Grant RGPIN-2021-02511 and NSERC Early Career Supplement DGECR-2021-00047. Email: \href{mailto:nmorrison@uvic.ca}{nmorrison@uvic.ca}.} 
 \and Shannon Ogden\footnotemark[3] \thanks{Supported by NSERC CGS M and NSERC Vanier CGS. Email: \href{mailto:sogden@uvic.ca}{sogden@uvic.ca}.}}
\date{\today}

\DeclareMathOperator{\inv}{inv}
\DeclareMathOperator{\mr}{mr}
\DeclareMathOperator{\tmr}{tmr}
\DeclareMathOperator{\rank}{rank}

\newcommand\F{\mathbb{F}}
\newcommand\M{\mathcal{M}}

\begin{document}

\maketitle
\begin{abstract}
    For an oriented graph $D$ and a set $X\subseteq V(D)$, the inversion of $X$ in $D$ is the graph obtained from $D$ by reversing the orientation of each edge that has both endpoints in $X$. Define the inversion number of $D$, denoted $\inv(D)$, to be the minimum number of inversions required to obtain an acyclic oriented graph from $D$. The dijoin, denoted $D_1\rightarrow D_2$, of two oriented graphs $D_1$ and $D_2$ is constructed by taking vertex-disjoint copies of $D_1$ and $D_2$ and adding all edges from $D_1$ to $D_2$. We show that $\inv({D_1 \rightarrow D_2}) > \inv(D_1)$, for any oriented graphs $D_1$ and $D_2$ such that $\inv(D_1) = \inv(D_2) \ge 1$. This resolves a question of Aubian, Havet, H\"orsch, Klingelhoefer, Nisse, Rambaud and Vermande. Our proof proceeds via a natural connection between the graph inversion number and the subgraph complementation number.
\end{abstract}


\section{Introduction}

Given an oriented graph $D$ and a set $X\subseteq V(D)$, the \emph{inversion} of $X$ in $D$ is the oriented graph obtained from $D$ by reversing the orientation of each edge that has both endpoints in $X$. In this case, we say that we \emph{invert} $X$ in $D$. Given a family of sets $X_1,...,X_k\subseteq V(D)$, the \emph{inversion} of $X_1,...,X_k$ in $D$ is the oriented graph obtained by inverting each set in turn: inverting $X_1$ in $D$, then $X_2$ in the resulting oriented graph, and so on. Note that the order in which we perform these inversions does not impact the final oriented graph. 

If inverting $X_1,...,X_k$ in $D$ produces an acyclic oriented graph, then these sets form a \emph{decycling family} of $D$.
The inversion number was introduced by Belkechine~\cite{belkhechine2009ind} and early results on the topic were obtained by Belkechine, Bouaziz, Boudabbous and Pouzet~\cite{belkhechine2010inversion}. 

Given oriented graphs $D_1$ and $D_2$, the \emph{dijoin} from $D_1$ to $D_2$, denoted by $D_1\rightarrow D_2$ is the oriented graph constructed from vertex-disjoint copies of $D_1$ and $D_2$ by adding all edges $uv$ where $u\in V(D_1)$ and $v\in V(D_2)$.  
Bang-Jensen, Costa Ferreira da Silva and Havet~\cite{bang_jensen2023inversion} observed that if $D_1$ and $D_2$ are oriented graphs then $\inv(D_1\rightarrow D_2) \le \inv(D_1) + \inv(D_2)$, 
and conjectured that equality holds for all $D_1,D_2$. They proved that the conjecture holds if $\inv(D_1) = \inv(D_2) =1$. 

However, this conjecture was shown to be false by two simultaneous papers~\cite{alon2024invertibility, aubian2022problems}. The authors of~\cite{aubian2022problems} provide a whole family of counterexamples, showing that for every odd $k \ge 3$ there is a tournament $D_1$ with $\inv(D_1) = k$ such that for any oriented graph $D_2$ with $\inv(D_2) \ge 1$, we have $\inv(D_1\rightarrow D_2) \le \inv(D_1) + \inv(D_2) - 1$. Thus the trivial upper bound on the inversion number of a dijoin is not always tight. 

For a trivial lower bound, it is easy to see that $\inv(D_1\rightarrow D_2) \ge \max\{\inv(D_1), \inv(D_2)\}$. As a first step towards investigating the tightness of this lower bound,  Aubian, Havet, H\"orsch, Klingelhoefer, Nisse, Rambaud and Vermande asked the following question.
\begin{question}[\protect{\cite[Problem 5.5]{aubian2022problems}}]
    Does there exist a non-acyclic oriented graph $D$ such that 
    \[\inv(D \rightarrow D) = \inv(D)?\]
\end{question}

We answer this question in the negative. In fact, we prove the following slightly more general result.
\begin{restatable}{theorem}{main}
\label{thm:D_dijoin_D_general}
    Let $D_1$ and $D_2$ be oriented graphs such that $\inv(D_1) = \inv(D_2) \ge 1$. Then $\inv({D_1 \rightarrow D_2}) > \inv(D_1)$.
\end{restatable}

In order to prove this theorem, we use a natural connection between the subgraph complementation number, as studied by Buchanan, Purcell and Rombach in~\cite{buchanan2022subgraph}, and the inversion number,
which we believe may be useful in future research on this topic.
This allows us to deduce that the inversion number of a digraph $D$ is either $\tmr(D)$ or $\tmr(D) + 1$, where $\tmr(D)$ is the minimum rank across a family of matrices (see Section~\ref{sec:subgraph_complement}).
The same connection is made in \cite{pouzet2023b}, however (by using the results in \cite{buchanan2022subgraph}) we are able to classify when $\inv(D) = \tmr(D) + 1$, which is a vital ingredient in our proof.
We discuss this further in Section~\ref{sec:subgraph_complement}, after first noting some easy observations in Section~\ref{sec:observations}. We prove Theorem~\ref{thm:D_dijoin_D_general} in Section~\ref{sec:proof}. Some open problems and conjectures are given in Section~\ref{sec:probs}.

\section{Preliminaries}
\label{sec:observations}

In this section, we recall some definitions and notation pertaining to oriented graphs, and present some basic results on the inversion number of oriented graphs which will be useful in later sections. 

Recall that an \emph{oriented} graph is a pair $D = (V,E)$, where $V$ is a collection of vertices and $E \subseteq V^{(2)}$ is a collection of ordered pairs of distinct vertices such that, for any $u,v \in V$ at most one of $uv$ and $vu$ is in $E$. For $u,v \in V$, we write $uv$ to denote the \emph{edge oriented (or directed) from u to v}.
An oriented graph can be viewed as the result of assigning a direction to, or orienting, each edge of a suitable simple graph.
A \emph{tournament} is an oriented graph where exactly one of $uv$ and $vu$ is present for all $u \not= v \in V$, i.e. an orientation of a complete graph.
An oriented graph $D_1$ is a \emph{subgraph} of $D_2$, denoted $D_1 \subseteq D_2$, if $V(D_1) \subseteq V(D_2)$ and $E(D_1) \subseteq E(D_2)$.

The \emph{out-neighbourhood} of a vertex $v$, denoted $N^{+}(v)$, is the set of all vertices $u\in V$ such that $vu\in E$. The \emph{in-neighbourhood} of $v$, denoted $N^{-}(v)$, is the set of all vertices $u\in V$ such that $uv\in E$. 
A vertex $v$ is a \emph{source} if  $N^{-}(v)$ is empty, and $v$ is a \emph{sink} if $N^{+}(v)$ is empty.

We first make the simple observation that removing vertices and edges from an oriented graph cannot increase the inversion number. Indeed, after removing the same vertices and edges from a decycling family of the initial oriented graph, it is a decycling family of the subgraph. 
\begin{observation}
\label{obs:inv subgraph}
    Let $D_1$ and $D_2$ be tournaments and suppose that $D_1\subseteq D_2$. Then $\inv(D_1)\leq \inv(D_2)$. 
\end{observation}

The following result can help to reduce a problem about oriented graphs to a problem about tournaments only.

\begin{proposition}\label{prop:digraph_to_tournament}
    For every oriented graph $D$, there is a tournament $D^*$ on the same vertex set with $D \subseteq D^*$ and $\inv(D^*) = \inv(D)$. 
\end{proposition}
\begin{proof}
    Let $k \coloneqq \inv(D)$.
    Let $U$ be the acyclic oriented graph that is reached from $D$ after applying a decycling family $X_1, X_2, \ldots, X_k$.
    There exists a transitive tournament $U^*$ on the same vertex set with $U \subseteq U^*$. Inverting the sets $X_1, X_2, \ldots, X_k$ in $U^*$ gives a tournament $D^*$ with $D \subseteq D^*$. Clearly $\inv(D^*) \le k$ and so by Observation \ref{obs:inv subgraph}, $\inv(D^*) = k$.
\end{proof}

We will also use the following simple results that consider the effect that the removal of a single vertex has on the inversion number of an oriented graph. 

\begin{proposition}
\label{prop:source/sink}
    Let $D$ be an oriented graph on $n\geq 2$ vertices. If $v\in V(D)$ is a sink or a source, then $\inv(D-v)=\inv(D)$. 
\end{proposition}

\begin{proof}
    By \Cref{obs:inv subgraph}, $\inv(D-v)\leq \inv(D)$. Now, since $v$ is a sink or a source, a decycling family of $D-v$ is also a decycling family of $D$. Hence, $\inv(D-v)\geq \inv(D)$.  
\end{proof}

Given an oriented graph $D$, we say that $u,v\in V(D)$ are \emph{twin} vertices if $N^+(u)\setminus\{v\}=N^+(v)\setminus\{u\}$ and $N^-(u)\setminus\{v\}=N^-(v)\setminus\{u\}$. 

\begin{proposition}
\label{prop:twin}
    Let $D$ be an oriented graph on $n\geq 2$ vertices. If $u,v\in V(D)$ are twin vertices, then $\inv(D-v)=\inv(D)$. 
\end{proposition}

\begin{proof}
    By \Cref{obs:inv subgraph}, $\inv(D-v)\leq \inv(D)$. Now, suppose $X_1,...,X_k$ is a decycling family of $D-v$. For $1\leq i\leq k$, let $Y_i=X_i$ if $u\notin X_i$, and $Y_i=X_i\cup \{v\}$ if $u\in X_i$. Then $Y_1,...,Y_k$ is a decycling family of $D$, and $\inv(D-v)\geq \inv(D)$.
\end{proof}

Finally, we have the following trivial upper bound on the inversion number of an oriented graph. 

\begin{proposition}
\label{prop:inv upper bound}
    For an oriented graph $D$ of order $n\geq 1$, we have $\inv(D) \le n-1$.
\end{proposition}

\begin{proof}
    Note that the inversion number of an oriented graph on one vertex is clearly zero, and assume the statement of the proposition is true for oriented graphs of order $n-1$.
    Let $v$ be a vertex in $D$, and let $X=N^+(v)\cup\{v\}$. Inverting $X$ in $D$, gives an oriented graph $D'$ in which $v$ is a sink. Hence, using Proposition~\ref{prop:source/sink}, $\inv(D') = \inv(D' - v) \leq n -2$, and $\inv(D) \leq \inv(D') + 1 \leq n -1$.
\end{proof}

\section{Subgraph complementation and tournament minimum rank}
\label{sec:subgraph_complement}
In this section, we discuss a natural connection between the subgraph complementation number, as studied in~\cite{buchanan2022subgraph}, and the inversion number. 
Importantly, we will show that the inversion number is closely related to the lowest rank of a matrix from a particular set of matrices, a key step in our proof of \Cref{thm:D_dijoin_D_general}.

\subsection{Background on subgraph complementation}
In order to state the results of Buchanan, Purcell and Rombach~\cite{buchanan2022subgraph}, we first require the following definitions.

Given an (undirected) graph $G$ of order $n\geq 1$ and a set $X\subseteq V(G)$,  the \emph{subgraph complementation} of $X$ in $G$ is the graph obtained from $G$ by complementing the edges in $G[X]$. In this case, we say that we \emph{complement} $X$ in $G$. Given a family of sets $X_1,...,X_k\subseteq V(G)$, the \emph{subgraph complementation} of $X_1,...,X_k$ in $G$ is the graph obtained by complementing each set in turn: complementing $X_1$ in $G$, then $X_2$ in the resulting graph, and so on. Note that the order in which we perform these subgraph complementations does not impact the final graph. 

If complementing $X_1,...,X_k$ in $G$ results in the empty graph $\overline{K_n}$, then these sets form a \emph{subgraph complementing system} of $G$. 
The \emph{subgraph complementation number} of $G$, denoted by $c_2(G)$, is the minimum number of sets in a subgraph complementing system of $G$. 

Note that $\mathcal{F}$ is a subgraph complementing system of $G$ if and only if each pair of adjacent vertices appears together in an odd number of sets in $\mathcal{F}$, while each pair of non-adjacent vertices appears together in an even number of sets in $\mathcal{F}$. 

Let $\M(G)$ be the collection of all $n\times n$ matrices with entries in $\{0,1\}$ that can be obtained from the adjacency matrix\footnote{Recall that the \emph{adjacency matrix} of $G$, denoted $A(G)$, is the $n \times n$ matrix such that $A_{i,j} = 1$ whenever $ij \in E(G)$ and 0 otherwise (including on the diagonal).} of $G$ by altering diagonal entries.
In this paper the rank of a matrix is always taken over $\F_2$ and we will refer to the rank of a matrix taken over $\F_2$ as simply the rank. Define the \emph{minimum rank} of a graph $G$, denoted by $\mr(G)$, to be the minimum rank of a matrix in $\M(G)$. 

Buchanan, Purcell, and Rombach~\cite{buchanan2022subgraph} showed that the quantities $\mr(G)$ and $c_2(G)$ cannot differ by more than $1$. In addition, they characterised the graphs for which they differ. 

\begin{lemma}[\cite{buchanan2022subgraph} Corollary 4.7 and Theorem 4.12] \label{lem:minrank}
    Let $G$ be a graph. Then either
    \begin{enumerate}
        \item $c_2(G) = \mr(G)$, or
        \item $c_2(G) = \mr(G) + 1$, in which case $\mr(G)$ is even.
    \end{enumerate}
    Moreover, if $G$ has at least one edge, then $c_2(G) = \mr(G) + 1$ if and only if there is a unique matrix $M \in \M(G)$ of minimum rank and all of the diagonal entries of this matrix are equal to zero.
\end{lemma}

Although it will not be directly relevant for our application of the result, the interested reader might like to know that the proof of Lemma~\ref{lem:minrank} uses an equivalent form of the problem. A $d$-dimensional \emph{faithful orthogonal representation} of a graph $G$ over the field $\F_2$ is a function $\phi:V(G)\rightarrow \F_2^d$ where non-adjacent vertices are assigned orthogonal vectors; that is, $\phi(u)\cdot \phi(v)= 0$ for all $uv\notin E(G)$, and adjacent vertices are assigned non-orthogonal vectors; that is, $\phi(u)\cdot \phi(v) = 1$. 
The $d$-dimensional faithful orthogonal representations of $G$ over $\F_2$ are in bijective correspondence with the subgraph complementation systems of $G$, where a representation $\phi$ corresponds to the system $\{X_1,...,X_d\}$ with $v\in V(G)$ included in $X_i$ if and only if the $i$th entry of $\phi(v)$ is $1$.
This approach bears a strong similarity to the analysis used in~\cite{aubian2022problems} to prove that for every odd $k \ge 3$ there is a tournament $D_1$ with $\inv(D_1) = k$ such that for any oriented graph $D_2$ with $\inv(D_2) \ge 1$, we have $\inv(D_1\rightarrow D_2) \le \inv(D_1) + \inv(D_2) - 1$. 

\subsection{Tournament minimum rank}

Let $D$ be a tournament on $n$ vertices and $T$ a transitive tournament on the same vertex set. Define $G_{D,T}$ to be the (undirected) graph on the same vertex set as $D$ with the edge $ij$ present if and only if the edge between vertices $i$ and $j$ has opposite orientations in $D$ and $T$.
Clearly, a series of inversions that takes $D$ to $T$ corresponds exactly to a subgraph complementing system of $G_{D,T}$.

Let $\mathcal{T}$ be the collection of all (labelled) $n$-vertex transitive tournaments and define 
\[\M^*(D) = \bigcup_{T \in \mathcal{T}} \M(G_{D,T}).\] Define the \emph{tournament minimum rank} of a tournament $D$, denoted $\tmr(D)$, as:
\[\tmr(D):= \min \{\rank(M): M \in \M^*(D)\}.\]
Equivalently, $\tmr(D) = \min_{T \in \mathcal{T}} \mr(G_{D,T})$.

The following result is a direct consequence of \Cref{lem:minrank} and provides a useful relationship between $\inv(D)$ and $\tmr(D)$.

\begin{corollary}[Corollary to Lemma \ref{lem:minrank}] \label{cor:tournament_minrank}
    Let $D$ be a tournament. Then either 
    \begin{enumerate}
        \item $\inv(D) = \tmr(D)$, or
        \item $\inv(D) = \tmr(D) + 1$, in which case $\tmr(D)$ is even.
    \end{enumerate}
    Moreover, if $D$ is not transitive, then $\inv(D) = \tmr(D) + 1$ if and only if every matrix $M \in \M^*(D)$ with minimum rank has every diagonal entry equal to zero.
\end{corollary}

% \begin{proof}
%     Let $T$ be an $n$-vertex transitive tournament. If transforming $D$ into $T$ requires $\ell$ inversions, then $c_2(G_{D,T}) = \ell$. So, by the definition of the inversion number,
%     \[\inv(D) = \min_{T \in \mathcal{T}} c_2(G_{D,T}).\]
%     Hence, by Lemma~\ref{lem:minrank} and the definition of minimum tournament rank, either 
%     \[\inv(D) = \min_{T \in \mathcal{T}}\mr(G_{D,T}) = \tmr(D),\]
%     or $\inv(D) = \tmr(D) + 1$ and $\tmr(D)$ is even. Moreover, if $D$ is not transitive, then $G_{D,T}$ contains at least one edge for any $T \in \mathcal{T}$ and Lemma~\ref{lem:minrank} also yields that $\inv(D) = \tmr(D) + 1$ if and only if every matrix of $D$ with minimum rank has zeroes on the diagonal.
% \end{proof}

\begin{proof}
        Let $T$ be an $n$-vertex transitive tournament. If transforming $D$ into $T$ requires $\ell$ inversions, then $c_2(G_{D,T}) = \ell$. So, by the definition of the inversion number,
        \[\inv(D) = \min_{T \in \mathcal{T}} c_2(G_{D,T}).\]
        By Lemma~\ref{lem:minrank}, we have that 
        $\mr(G_{D,T}) \leq c_2(G_{D,T}) \leq \mr(G_{D,T}) + 1$,
        and so
        \[\tmr(D) = \min_{T \in \mathcal{T}} \mr(G_{D,T}) \leq \min_{T \in \mathcal{T}} c_2(G_{D,T}) \leq \min_{T \in \mathcal{T}} \mr(G_{D,T}) + 1 = \tmr(D) + 1.\]
        Furthermore, the equality $\inv(D) = \tmr(D) + 1$ holds if and only if $c_2(G_{D,T}) = \mr(G_{D,T}) + 1$ for every $T \in \mathcal{T}$ with $\mr(G_{D,T}) = \tmr(D)$. This immediately implies that, if $\inv(D) = \tmr(D) + 1$, then $\tmr(D)$ must be even.

        Moreover, if $D$ is not transitive, then $G_{D,T}$ contains at least one edge for any $T \in \mathcal{T}$ and Lemma~\ref{lem:minrank} tells us that, if $c_2(G_{D,T}) = \mr(G_{D,T}) + 1$, then there is a unique matrix of minimum rank in $\M(G_{D,T})$ and it has all zeroes on the diagonal. Hence, if $\inv(D) = \tmr(D) + 1$, this is true of every $T$ with $\mr(G_{D,T}) = \tmr(D)$ and every matrix in $\M^*(D)$ of minimum rank has zeros on the diagonal.
\end{proof}

It is interesting to note that all of the examples in~\cite{alon2024invertibility,aubian2022problems} of pairs of graphs $D_1,D_2$ with $\inv(D_1\rightarrow D_2) < \inv(D_1) + \inv(D_2)$ have that $D_i$ is a tournament with $\inv(D_i) = \tmr(D_i) + 1$ for at least one $i$. In fact, the following result holds.
\begin{theorem} \label{thm:generating_examples}
    Let $D_1$ be a tournament with $\inv(D_1) = \tmr(D_1) + 1$, and let  $D_2$ be any oriented graph with $\inv(D_2) \ge 1$. Then $\inv(D_1\rightarrow D_2) \le \inv(D_1) + \inv(D_2) - 1$.
\end{theorem}
This can be proved using a similar argument to that used in~\cite{aubian2022problems}. Alternatively, we can directly apply Corollary~\ref{cor:tournament_minrank}, as follows.
\begin{proof}
    Let $D_1$ be a tournament satisfying $\inv(D_1) = \tmr(D_1) + 1$. Applying Proposition~\ref{prop:digraph_to_tournament}, let $D_2^*$ be a tournament containing $D_2$ with $ \inv(D_2^*)  =  \inv(D_2) $. We see that
\begin{align*}
    \inv(D_1 \rightarrow D_2) \le   \inv(D_1 \rightarrow D_2^*) &\le \tmr(D_1 \rightarrow D_2^*) + 1 \\
    &\le \tmr(D_1) + \tmr(D_2^*) + 1 \\
    &\le \inv(D_1) + \inv(D_2^*) = \inv(D_1) + \inv(D_2).
\end{align*}
Suppose for a contradiction that we have equality. Then
\begin{align}
    \inv(D_1 \rightarrow D_2^*) &= \tmr(D_1 \rightarrow D_2^*) + 1, \label{eqn:1}\\
    \tmr(D_1 \rightarrow D_2^*) &= \tmr(D_1) + \tmr(D_2^*), \quad \text{and} \label{eqn:2}\\
    \inv(D_2^*) &= \tmr(D_2^*). \label{eqn:3}
\end{align}
Using~\eqref{eqn:3} and $\inv(D_2^*) \ge 1$ (so $D_2^*$ is not transitive), Corollary~\ref{cor:tournament_minrank} tells us that there is some minimum rank matrix $M_2 \in \M^*(D_2^*)$ with a non-zero diagonal entry. Then, letting $M_1$ be a minimum rank matrix in $\M^*(D_1)$, the matrix
\[
\begin{bmatrix}
    M_1 & 0 \\ 0 & M_2
\end{bmatrix}
\]
is a matrix in $\M^*(D_1 \rightarrow D_2^*)$ with a non-zero entry on the diagonal and rank $\tmr(D_1) + \tmr(D_2^*)$, which is equal to $ \tmr(D_1 \rightarrow D_2^*)$ by~\eqref{eqn:2}. Applying Corollary~\ref{cor:tournament_minrank} again, this implies that $\inv(D_1 \rightarrow D_2^*) = \tmr(D_1 \rightarrow D_2^*)$, contradicting~\eqref{eqn:1}.
\end{proof}

\section{Proof of Theorem~\ref{thm:D_dijoin_D_general}}
\label{sec:proof}
In Problem 5.5  of \cite{aubian2022problems}, Aubian, Havet, H\"{o}rsch, Klingelhoefer, Nisse, Rambaud, and Vermande ask whether there exists a non-acyclic oriented graph $D$ such that $\inv({D \rightarrow D}) = \inv(D)$. 
The goal of this section is to prove Theorem~\ref{thm:D_dijoin_D_general}, and answer this question in the negative. We restate it below for convenience.

\main*

In fact, we may focus our attention exclusively on the tournament case.

\begin{lemma}
\label{lem:D_dijoin_D}
    Let $D_1$ and $D_2$ be tournaments such that $\inv(D_1) = \inv(D_2) \ge 1$. Then $\inv({D_1 \rightarrow D_2}) > \inv(D_1)$.
\end{lemma}

Before we prove Lemma \ref{lem:D_dijoin_D}, we demonstrate why this suffices to prove \Cref{thm:D_dijoin_D_general}.

\begin{proof}[Proof of Theorem \ref{thm:D_dijoin_D_general}] 
Suppose for a contradiction that there exist oriented graphs $D_1, D_2$ with $\inv(D_1) = \inv(D_2) =  \inv(D_1\rightarrow D_2)$. Let $k \coloneqq \inv(D_1)$. 

Apply Proposition \ref{prop:digraph_to_tournament} to obtain a tournament $(D_1 \rightarrow D_2)^*$ with $({D_1\rightarrow D_2}) \subseteq ({D_1\rightarrow D_2})^*$ and
\[\inv((D_1\rightarrow D_2)^*) = \inv(D_1\rightarrow D_2) = k.\]
Since $(D_1\rightarrow D_2)^*$  contains every edge of $D_1 \rightarrow D_2$, it is the dijoin of two tournaments $E_1$ and $E_2$, where $E_1 \supseteq D_1$ and  $E_2 \supseteq D_2$. By Observation \ref{obs:inv subgraph}, both $E_1$ and $E_2$ have inversion number at least $k$. Hence  
\[k = \inv((D_1\rightarrow D_2)^*) = \inv(E_1 \rightarrow E_2) \ge \inv(E_1) \ge k\]
and 
\[k = \inv((D_1\rightarrow D_2)^*)= \inv(E_1 \rightarrow E_2) \ge \inv(E_2) \ge k.\]
Therefore, we have two tournaments $E_1,E_2$ with $\inv(E_1) = k = \inv(E_2) = \inv(E_1 \rightarrow E_2)$, contradicting Lemma \ref{lem:D_dijoin_D}.
\end{proof}

In order to prove \Cref{lem:D_dijoin_D}, we will require the following lemma about the structure of certain symmetric matrices. Call an $n \times m$ matrix with entries in $\{0,1\}$ a \emph{staircase matrix} if its entries increase down each column and decrease along each row (so that the $1$s form the shape of a staircase in the bottom left).

\begin{lemma} \label{lem:block_matrices}
    Let $M$ be a symmetric $(n+m) \times (n+m)$ matrix with entries in $\{0,1\}$ of the form 
    \[
        \begin{bmatrix}
        A & C \\ C^T & B
        \end{bmatrix}
    \]
    where $A$ is a symmetric $n \times n$ matrix, $B$ is a symmetric $m \times m$ matrix and $C$ is an $n \times m$ staircase matrix.
    If $m \ge \rank(A) + 1$, then one of the following holds: 
    \begin{enumerate}
        \item  $\rank(M) \ge \rank(A) + 1$, or
        \item there are two adjacent columns of $B$ which are identical, or
        \item the final column of $B$ contains only zeroes.
    \end{enumerate}   
\end{lemma}
\begin{proof}
    Suppose that $M$ is a matrix of the given form.  Clearly $\rank(M) \ge \rank(A)$, so suppose that $\rank(M) = \rank(A) = k$. It must therefore be possible to write each row of $B$ as a linear combination of rows of $C$ over $\F_2$.
    
    Since  $\rank(M) = k$, it follows that the staircase $C$ must contain at most $k$ distinct non-zero columns (`steps'). Since $m \ge k + 1$, this means that either $C$ contains two consecutive columns with the same entries or $C$ contains a zero column. We split into two cases.

    \begin{case} Suppose that $C$ contains two adjacent columns with the same entries, say column $i$ and $i+1$. Since each row of $B$ can be written as a linear combination of rows of $C$ over $\F_2$, this implies that columns $i$ and $i+1$ of $B$ must also contain the same entries.
    \end{case}

    \begin{case} Suppose that $C$ contains a zero column. Since $C$ is a staircase matrix, the final column of $C$ must be a zero column. Since each row of $B$ can be written as a linear combination of rows of $C$, this implies that the final column of $B$ also contains only zeroes.
    \end{case}
\end{proof}

We are now armed with all the tools necessary to prove Lemma \ref{lem:D_dijoin_D}.

\begin{proof}[Proof of Lemma \ref{lem:D_dijoin_D}]
    Suppose for a contradiction that there exist non-transitive tournaments $D_1$ and $D_2$ with $\inv(D_1) = \inv(D_2) = \inv({D_1\rightarrow D_2})$. Take $D_1,D_2$ to be tournaments with this property such that $|V(D_1)| + |V(D_2)|$ is minimal, and let $n_1 := |V(D_1)|$ and $n_2 := |V(D_2)|$. Let $n:=n_1+n_2$ and $k \coloneqq \inv(D_1)$. Note that Proposition \ref{prop:inv upper bound} tells us that $k+1 \le n_1, n_2$. 

    Our goal is to show that every minimum rank matrix in $\M^*(D_1 \rightarrow D_2)$ has rank $k$ and only zero entries on the diagonal. Then, by Corollary \ref{cor:tournament_minrank}, we are able to conclude that $\inv({D_1 \rightarrow D_2}) = k+1$ to obtain a contradiction.

     Suppose that $M$ is a matrix of minimum rank in $\M^*(D_1 \rightarrow D_2)$, where $M\in \M(G_{D_1\rightarrow D_2,T})$ for some transitive tournament $T$. Note that $T$ naturally induces an order $\prec$ on its vertices, with the source as the first vertex and the sink as the last vertex.
    We fix a different ordering $\phi:V(T)\rightarrow[n]$ of the vertices of $T$ (and thus of $D_1 \rightarrow D_2$), which is obtained by first taking all vertices in the copy of $D_1$ in the order induced by the natural ordering on $T$, and then taking all vertices in the copy of $D_2$ in the order induced by $T$. That is, $\phi(u) < \phi(v)$ if and only if either $u\in V(D_1)$ and $v\in V(D_2)$, or $u,v\in V(D_1)$ and $u\prec v$, or $u,v\in V(D_2)$ and $u \prec v$.
    
    Note that permuting both the rows and the columns of $M$ by a given permutation does not change the rank of $M$ or the diagonal entries. Therefore, we may, and will, assume that our matrix $M$ has rows and columns ordered according to the vertex ordering $\phi$.   
    
    Now, since $M$ has minimum rank, by \Cref{cor:tournament_minrank}, $\rank(M) \in \{k-1,k\}$.
    Moreover, by our choice of vertex order, $M$ has the form
    \begin{equation}\label{eq:rank}
        \begin{bmatrix}
        A & C \\ C^T & B
        \end{bmatrix}
    \end{equation}
    where $A$ is an $n_1 \times n_1$ symmetric matrix, $B$ is an $n_2 \times n_2$ symmetric matrix and $C$ is an $n_1 \times n_2$ staircase matrix. 
    To see that $C$ is indeed a staircase matrix, consider a 1 in $C$, and suppose it corresponds to the edge $uv$ where $u \in V(D_1)$ and $v \in V(D_2)$. Since this entry, which we denote $C_{u,v}$, is a 1, we have $v \prec u$. An entry $C_{u,v'}$ to the left of $C_{u,v}$ corresponds to an edge between $u$ and some vertex $v' \in V(D_2)$ with $v' \prec v$. Hence, $v' \prec v \prec u$ and the entry $C_{u,v'}$ must also be a 1. Similarly an entry below $C_{u,v}$ corresponds to the edge between $v$ and some vertex $u'$ with $u \prec u'$, it must also be a 1 as we have $v \prec u \prec u'$.
    
    %In particular, $A \in \M^*(D_1)$ and $B\in \M^*(D_2)$. 
    
    Clearly $\rank(A) \le \rank(M) \le k$. 
    Since $A \in \M^*(D_1)$, by \Cref{cor:tournament_minrank},
    \[\rank(A) \ge \tmr(D_1) \ge \inv(D_1) -1 = k-1.\] The corresponding inequalities also hold for $B$, and thus $\rank(A), \rank(B) \in \{k-1,k\}$. 

    \begin{claim}\label{cl:rank}
        $\rank(A) = \rank(B) = k-1$ and $\rank(M) = k$.
    \end{claim}

    \begin{proof}
    First suppose, in order to obtain a contradiction, that $\rank(A) = \rank(M)$. Since $n_2 \ge k+1 \ge \rank(A) + 1$, by Lemma \ref{lem:block_matrices} we can immediately deduce that either there are two adjacent columns of $B$ that have the same entries, or the final column of $B$ contains only zeroes. 

    Suppose there are two adjacent columns of $B$ with the same entries, and let these correspond to the vertices $u$ and $v$. Let $i = \phi(u)$ and note that $\phi(v) = i + 1$. By definition of $\phi$, a vertex $w \in D_2$ is in $N^+(u)$ if and only if either $\phi(w) < i$ and $B_{\phi(w),i} = 1$, or $\phi(w) > i$ and $B_{\phi(w),i} = 0$. Similarly, a vertex $w \in D_2$ is in $N^+(v)$ if and only if either $\phi(w) < i+1$ and $B_{\phi(w),i+1} = 1$, or $\phi(w) > i+1$ and $B_{\phi(w),i+1} = 0$. Since $B_{j,i} = B_{j,i+1}$ for all $j$, we see that $N^+(u)\setminus\{v\} = N^+(v)\setminus\{u\}$. In particular, $u$ and $v$ are twin vertices in $D_2$.
    Let $D_2' = D_2 - u$.
    By Proposition \ref{prop:twin}, $\inv(D_2') = \inv(D_2)$.

    Otherwise, suppose that the final column of $B$ is all zeros. This means that the vertex $u$ with $\phi(u)=n$ is a sink in $D_2$. Let $D_2' = D_2-u$. By Proposition \ref{prop:source/sink},  $\inv(D_2') = \inv(D_2)$.

    In either case, by Observation \ref{obs:inv subgraph}, 
    \[k = \inv(D_2) = \inv(D_2') \le \inv({D_1 \rightarrow D_2'}) \le \inv({D_1 \rightarrow D_2}) = k.\]
    In particular, $\inv(D_2') = \inv({D_1 \rightarrow D_2'})=k$ and $D_2'$ has one fewer vertex than $D_2$, contradicting the minimality of $|V(D_1)| + |V(D_2)|$. Hence, $\rank(A)<\rank(M)$.      
   
   Now suppose $\rank(B) = \rank(M)$. This follows along the same lines as the previous case, the only difference being that we must apply Lemma \ref{lem:block_matrices} to $M$ with rows and columns in reverse. Either there are two adjacent columns of $A$ that have the same entries or the first column of $A$ contains only zeroes, corresponding to $D_1$ containing twin vertices or a source vertex, respectively. The proof then proceeds as before, and thus $\rank(B) < \rank(M)$.  Therefore, $\rank(A) = \rank(B) = k-1$ and $\rank(M) = k$.
   \end{proof}

By Claim~\ref{cl:rank}, $\inv(D_1) = \rank(A) + 1$ and $\inv(D_2) = \rank(B) + 1$. Hence, by Corollary~\ref{cor:tournament_minrank}, $A$ and $B$ (and thus $M$) must have zero entries on the diagonal. Therefore, every matrix of $\M^*(D_1 \rightarrow D_2)$ of minimum rank has rank $k$ and every diagonal entry equal to zero. By Corollary \ref{cor:tournament_minrank}, we conclude that $\inv({D_1 \rightarrow D_2}) = k+1$, which is a contradiction.
\end{proof}

\section{Open problems}\label{sec:probs}

In light of Theorem~\ref{thm:generating_examples}, and the fact that 
all of the examples in~\cite{alon2024invertibility,aubian2022problems} of pairs of oriented graphs $D_1,D_2$ with $\inv(D_1\rightarrow D_2) < \inv(D_1) + \inv(D_2)$ can be obtained by an application of this theorem, we ask whether these are all such examples.

\begin{question}\label{q:other_examples}
    Do there exist tournaments $D_1, D_2$ with $\inv(D_i) = \tmr(D_i)$ for $i = 1,2$ and $\inv(D_1\rightarrow D_2) < \inv(D_1) + \inv(D_2)$?
\end{question}

Note that by \Cref{cor:tournament_minrank}, for any tournament $D$, if $\inv(D)$ is even, then $\inv(D) = \tmr(D)$. Hence, a negative answer to this question would disprove the following pair of similar conjectures (the latter of which is strictly stronger than the former). 
\begin{conjecture}[\protect{\cite[Conjecture 8.9]{alon2024invertibility}}]
    For all $\ell ,r\in \mathbb{N}$ with $\ell \geq 3$ or $r\geq 3$ there exist oriented graphs $D_1$ and $D_2$ with $\inv(D_1) =\ell$ and $\inv(D_2) =r$, but $\inv(D_1\rightarrow D_2)< \inv(D_1) + \inv(D_2)$.
\end{conjecture}
\begin{conjecture}[\protect{\cite[Conjecture 5.3]{aubian2022problems}}]
    For all $\ell \geq 3$ there exists an oriented graph $D_1$ with $\inv(D_1) =\ell$ such that for all $D_2$ with $\inv(D_2) \ge 1$, we have $\inv(D_1\rightarrow D_2)< \inv(D_1) + \inv(D_2)$.
\end{conjecture}

One approach to answering Question~\ref{q:other_examples} would be to bound the tournament minimum rank of the dijoin of two tournaments.
\begin{question} \label{q:tmr_adds}
    Do there exist tournaments $D_1, D_2$ with $\tmr(D_1 \rightarrow D_2) < \tmr(D_1) + \tmr(D_2) $?
\end{question}
A negative answer to this question would be a very strong result that would also answer  Question~\ref{q:other_examples} in the negative, and therefore resolve the two conjectures above. Moreover, we could immediately conclude that \[
\inv(D_1) + \inv(D_2) - 2 \le \inv(D_1 \rightarrow D_2) \le  \inv(D_1) + \inv(D_2).\]
for all oriented graphs $D_1,D_2$, by a simple application of Proposition~\ref{prop:digraph_to_tournament} and Corollary~\ref{cor:tournament_minrank}.

One place to start would be to answer Question~\ref{q:tmr_adds} in the special case when one of the tournaments is $\overrightarrow{C_3}$, the directed cycle on three vertices.

\begin{conjecture}
    For all tournaments $D$, we have $\tmr(D \rightarrow \overrightarrow{C_3}) = \tmr(D) + 1 =  \tmr( \overrightarrow{C_3} \rightarrow D)$.
\end{conjecture}

We can generalise the idea of dijoins to sequences of graphs. Given a finite sequence $D_1,...,D_k$ of oriented graphs, the $k$\emph{-join} of $D_1,...,D_k$, denoted by $[D_1,..., D_k]$, is the oriented graph constructed from vertex-disjoint copies of $D_1,...,D_k$ by adding all edges $uv$ where $u\in V(D_i)$ and $v\in V(D_j)$ for $i<j$. For ease of notation, we write $[D]_k$ for the $k$-join of $k$ copies of the same oriented graph $D$. 

Pouzet, Kaddour and Thatte~\cite{pouzet2023b} proved that $\inv\left(\left[\overrightarrow{C_3}\right]_k\right) = k$ for all $k$, where $\overrightarrow{C_3}$ is the directed cycle on three vertices. Further to this, Alon, Powierski, Savery, Scott and Wilmer~\cite{alon2024invertibility} proved that if $D_1,D_2,\ldots,D_k$ are oriented graphs with $\inv(D_i) \le 2$ for all $i$ and  $\inv(D_i) = 2$ for at most one $i$, then \[
\inv([D_1,D_2, \ldots , D_k]) = \sum_{i=1}^k \inv(D_i).
\]

They conjecture that the condition that $\inv(D_i) = 2$ for at most one $i$ is unnecessary.
\begin{conjecture}[\protect{\cite[Conjecture 8.8]{alon2024invertibility}}]
 Let $k \in \mathbb{N}$,  and  let $D_1,\ldots,D_k$ be  oriented  graphs  satisfying $\inv(D_i)\leq 2$ for all $i$.  Then $\inv([D_1,\ldots,D_k]) =\sum_{i=1}^k \inv(D_i)$.
\end{conjecture}
We remark that $\inv(D) = \tmr(D)$ for every tournament with $\inv(D) \leq 2$, and so a negative answer to Question~\ref{q:tmr_adds} would immediately lead to a proof of this conjecture. In addition, combined with Corollary~\ref{cor:tournament_minrank}, it would give an affirmative answer to the following question, yielding a more general result.
\begin{question}
 Let $k \in \mathbb{N}$, and let $D_1,\ldots,D_k$ be oriented graphs such that, for every $i$, either $\inv(D_i) =1$ or $\inv(D_i)$ is even.  Is  $\inv([D_1,\ldots,D_k]) =\sum_{i=1}^k \inv(D_i)$?
\end{question}

\paragraph{Acknowledgements} We would like to thank the anonymous referees for their helpful comments.

\bibliographystyle{plainurl}
\bibliography{bib}

\begin{thebibliography}{1}

\bibitem{alon2024invertibility}
Noga Alon, Emil Powierski, Michael Savery, Alex Scott, and Elizabeth Wilmer.
\newblock Invertibility of digraphs and tournaments.
\newblock {\em SIAM J. Discrete Math.}, 38(1):327--347, 2024.
\newblock \href {https://doi.org/10.1137/23M1547135}
  {\path{doi:10.1137/23M1547135}}.

\bibitem{aubian2022problems}
Guillaume Aubian, Frédéric Havet, Florian Hörsch, Felix Klingelhoefer,
  Nicolas Nisse, Clément Rambaud, and Quentin Vermande.
\newblock Problems, proofs, and disproofs on the inversion number, 2022.
\newblock \href {http://arxiv.org/abs/2212.09188} {\path{arXiv:2212.09188}}.

\bibitem{bang_jensen2023inversion}
J{\o}rgen Bang-Jensen, Jonas Costa Ferreira~da Silva, and Fr\'{e}d\'{e}ric
  Havet.
\newblock On the inversion number of oriented graphs.
\newblock {\em Discrete Math. Theor. Comput. Sci.}, 23(2):Paper No. 8, 29,
  2022.
\newblock \href {https://doi.org/10.46298/dmtcs.7474}
  {\path{doi:10.46298/dmtcs.7474}}.

\bibitem{belkhechine2009ind}
Houmem Belkhechine.
\newblock {\em {Ind{\'e}composabilit{\'e} des graphes et des tournois}}.
\newblock Theses, {Universit{\'e} Claude Bernard Lyon I; Universit{\'e} de
  Sfax. Facult{\'e} des sciences}, July 2009.
\newblock URL: \url{https://theses.hal.science/tel-00609544}.

\bibitem{belkhechine2010inversion}
Houmem Belkhechine, Moncef Bouaziz, Imed Boudabbous, and Maurice Pouzet.
\newblock Inversion dans les tournois.
\newblock {\em C. R. Math. Acad. Sci. Paris}, 348(13-14):703--707, 2010.
\newblock \href {https://doi.org/10.1016/j.crma.2010.06.022}
  {\path{doi:10.1016/j.crma.2010.06.022}}.

\bibitem{buchanan2022subgraph}
Calum Buchanan, Christopher Purcell, and Puck Rombach.
\newblock Subgraph complementation and minimum rank.
\newblock {\em Electron. J. Combin.}, 29(1):Paper No. 1.38, 20, 2022.
\newblock \href {https://doi.org/10.37236/10383} {\path{doi:10.37236/10383}}.

\bibitem{pouzet2023b}
Maurice Pouzet, Hamza Si~Kaddour, and Bhalchandra~D. Thatte.
\newblock On the {B}oolean dimension of a graph and other related parameters.
\newblock {\em Discrete Math. Theor. Comput. Sci.}, 23(2):Paper No. 5, 20,
  2022.
\newblock \href {https://doi.org/10.46298/dmtcs.7437}
  {\path{doi:10.46298/dmtcs.7437}}.

\end{thebibliography}
\end{document}